\documentclass[12pt]{article}

\usepackage{amsmath} 
\usepackage{amssymb} 
\usepackage{amsthm} 
\usepackage{graphicx} 
\usepackage{verbatim} 
\usepackage{color} 
\usepackage{subfigure} 
\usepackage{hyperref} 

\theoremstyle{definition} \newtheorem{defn}{Definition}
\theoremstyle{definition} \newtheorem{thm}{Theorem}
\theoremstyle{definition} 
\theoremstyle{definition} \newtheorem*{remark}{Remark}
\theoremstyle{plain} \newtheorem{example}{Example}
\theoremstyle{definition} \newtheorem*{cor}{Corollary}
\theoremstyle{definition} \newtheorem{construction}{Construction}

\begin{document}

\author{D.Fletcher}
\title{Aperiodic tilings with one prototile and low complexity atlas matching rules}
\maketitle

\begin{abstract}
We give a constructive method that can decrease the number of prototiles needed to tile a space.
We achieve this by exchanging edge to edge matching rules for a small atlas of permitted patches.
This method is illustrated with Wang tiles, and we apply our method to present via these rules a single prototile that can only tile $\mathbb{R}^3$ aperiodically, and a pair of square tiles that can only tile $\mathbb{R}^{2}$ aperiodically.

\end{abstract}

\section{Introduction}
The field of aperiodic tilings was created by Berger's discovery \cite{BERGER} of a set of $20426$ square tiles which were \emph{strongly aperiodic} in the sense that they could only tile the plane in a non-repeating global structure. Unsurprisingly there has been some interest on how far this number could be decreased. The number of tiles was reduced over time, to $6$ square tiles by Robinson \cite{ROBIN} in $1971$, and, relaxing to non-square tiles, to $2$ tiles by Penrose \cite{PENROSE} two years later. This led naturally to serious consideration about the possible existence of a single tile.
While a simple example has not been forthcoming, if we relax the requirement that the monotile be completely defined by its shape alone, there has been progress.

In \cite{SOCOLAR} Socolar studied a more general problem, \lq $k$-isohedral' monotiles, which had the monotile as a limiting case. Relaxing conditions on edge-coloring, non-connected tiles or space-filling provided positive results, but not in the limit.
In 1996 Gummelt \cite{GUMMELT} considered tiles that where allowed to overlap, and produced a decorated tile which could force strong aperiodicity.

This paper extends work by Goodman-Strauss on \lq atlas matching rules'. In \cite{STRAUSSSURVEY} Goodman-Strauss describes how by
requiring a tiling be covered by a suitable finite atlas of permitted bounded configurations, a domino can serve as a monotile. Sadly the atlas requires patches of extremely large radius.
This paper describes a method of altering matching rules from coloured tiles to atlas matching rules with very small patches. Furthermore if two of the tiles have the same shape, the number of prototiles needed is decreased.

We use this method to construct a pair of square tiles which tile $\mathbb{R}^{2}$ aperiodically, and a single cubic tile that tiles $\mathbb{R}^{3}$ aperiodically.

\section{The atlas matching rule construction}

We shall describe the basic definitions we will be using in this paper, drawing many of them from \cite{COHOMOBDHS}.
For clarity we will be limiting the spaces we are tiling to $\mathbb{R}^{n}$ for some $n \in \mathbb{N}$. With minor alterations the method will work in any homogeneous space (for example hyperbolic space $\mathbb{H}^{n}$).

Let $\mathcal{P}$ be a finite set of compact subsets of $\mathbb{R}^{n}$, each the closure of its interior. Denote these subsets as \emph{prototiles}.
Let $G$ be a group of isometries of $\mathbb{R}^{n}$, which includes all translations of $\mathbb{R}^{n}$. The groups we will be using most in this paper are the group of translations $G_{Tr}$ and the group of all isometries $G_{I}$.
For a given set of prototiles $\mathcal{P}$ and a group of isometries $G$, define a tile $t$ as the image $f(P)$, for some $f \in G$, $P \in \mathcal{P}$.
A \emph{patch} for $\mathcal{P}$ is a set of tiles with pairwise disjoint interiors and the \emph{support} of a patch
is the union of its tiles. A \emph{tiling} with prototiles $\mathcal{P}$ is a patch with
support $\mathbb{R}^{n}$. We shall refer to the support of a prototile $P$ as $supp(P)$.

We now want to introduce the notion of \lq decorating' a prototile, and hence all tiles produced from it.
Construct a function $g:\coprod_{P \in \mathcal{P}} P \to C$ where $C$ is a set containing a distinguished element, $0$ say, and possibly other elements.
A point $x$ in the prototile $P$ is \emph{c-coloured} if $g(x)=c$. We will refer to points that are $0$-coloured as \emph{uncoloured} points.

Extend $g$ to points of any given tile $t=f(P)$ by $g_{t}(x) = g f^{-1}(x)$ for each $x \in t$.

\begin{defn}
A coloured tiling $(T,g)$ satisfies the \emph{identical facet (matching) rule} if for all tiles $t_{1}$, $t_{2}$ (where $t_{1} \neq t_{2}$), $g_{t_{1}}(x)=g_{t_{2}}(x)$ for all $x \in t_{1} \cap t_{2}$.

\end{defn}

This covers cases where two tiles \lq match' if they have the same colour on the interior of their shared boundary (for example Wang tiles).
We will be using a slightly more general version of this rule in the rest of this paper, which allows tiles to match under wider conditions, as follows.

\begin{defn}
An \emph{facet (matching) rule} is a function $r: C \times C \mapsto \{0,1\}$ such that $r(x,y) = r(y,x)$ and $r(0,0)=1$.
A coloured tiling $(T,g)$ satisfies the \emph{facet (matching) rule} $r$ if for all tiles $t_{1}$, $t_{2}$ (where $t_{1} \neq t_{2}$),
\\
\\
 $r(g_{t_{1}}(x), g_{t_{2}}(x))=1$ for all $x \in t_{1} \cap t_{2}$.
\end{defn}

We describe below a way of translating from this style of matching rule to the following matching rule.

\begin{defn}
A tiling $T$ satisfies an \emph{atlas (matching) rule} $\mathcal{U}$ if there exists an atlas of patches $U \in \mathcal{U}$ such that for every tile $t \in T$, there exists a patch $n(t)$ about $t$ (with $t$ being in the strict interior of $n(t)$) such that $n(t)$ is a translation of some $U \in \mathcal{U}$.

Furthermore $\mathcal{U}$ must have a finite number of elements, and any patch in $\mathcal{U}$ must be compact.

A prototile set $\mathcal{P}$ satisfies the \emph{atlas (matching) rule} $\mathcal{U}$ if all tilings with prototiles $P$ satisfy the atlas matching rule $\mathcal{U}$.

In this paper, we will be using patches defined by the \lq 1-corona' about a tile $t$.
The \lq 1-corona' of a tile $t$ is the set of tiles touching $t$ (see \cite{STRAUSSSURVEY}).


\end{defn}

\begin{defn}
A tiling $T$ is a $(\mathcal{P},G,g,r)$\emph{-tiling} if it has a prototile set $\mathcal{P}$ with allowable isometries $G$ and colouring $g$, and satisfies the facet rule $r$.

A tiling $T$ is a $(\mathcal{X},G,\mathcal{U})$\emph{-tiling} if it has a prototile set $\mathcal{X}$ with allowable isometries $G$, and satisfies the atlas matching rule $\mathcal{U}$.

Two tilings are \emph{MLD} (mutually locally derivable) if one is obtained from the other in a unique way by local rules, and vice versa.
\end{defn}

\begin{thm}
A $(\mathcal{P},G_{Tr},g,r)$-tiling $T$ is MLD to a $(\mathcal{X},G_{I},\mathcal{U})$-tiling for some $1$-corona atlas rule $\mathcal{U}$ and a prototile set $\mathcal{X}$ with $|\mathcal{X}| \leq |\mathcal{P}|$.
\end{thm}

\begin{construction}
Take $\mathcal{P}$ and partition it into a set of equivalence classes $\mathcal{P}= \coprod \mathcal{P}_{s}$, $s \in \{1, \ldots, m\}$ where $P_{i} \sim P_{j}$ iff $\emph{supp}(P_{i})=\emph{supp}(P_{j})$ up to the action of an element of $G_{Tr}$.
For all $\mathcal{P}_{s}$, let $G_{s}$ be the largest subgroup of $G_{I} / G_{Tr}$ such that for all $f \in G_{s}$ and all $P \in \mathcal{P}_{s}$, $f(P) \sim P$.

Enumerate the elements of $\mathcal{P}_{s}$ as $P_{1}^{s}, \ldots, P_{r}^{s} \in \mathcal{P}_{s}$.

Choose the smallest $k$ you can so as to construct an injective function $e_{s}:\mathcal{P}_{s} \to \{(P_{i}^{s} , g_{s}) |  1 \leq i \leq k , g_{s} \in G_{s}\}$.
Define $\mathcal{X}_{s} = \{P_{1}^{s}, \ldots , P_{k}^{s} \}$.
\end{construction}

We now have a construction taking prototiles $P_{i}^{s} \in \mathcal{P}_{s}$ to ordered pairs of a prototile from $\mathcal{X}_{s}$ and an automorphism of that prototile. Observe that $\mathcal{X}_{s}$ is a subset of $\mathcal{P}_{s}$.
\\
\\
\emph{Proof of Theorem 1} 

Define a new prototile set $\mathcal{X}=\mathcal{X}_{1}\cup \ldots \cup \mathcal{X}_{m}$, where $\mathcal{X}_{s}$ is as just defined. Let the set of allowable functions from the prototiles into $\mathbb{R}^{n}$ be $G_{I}$, instead of $G_{Tr}$.
Take the set of allowable $1$-coronas in the $(\mathcal{P},G_{Tr},g,r)$-tiling $T$, and replace every tile originating from a translation of a prototile $P_{i}^{s} \in \mathcal{P}_{s}$ with $g_{s}(P_{j}^{s})$, with $g_{s}$ and $P_{j}^{s}$ originating from $e_{s}(P_{i}^{s})=(g_{s},P_{j}^{s})$.
This will give you a set of $1$-corona patches of $\mathcal{X}$. Use this set as the atlas rule $\mathcal{U}$ for $\mathcal{X}$.

$T$ has facet rules, which are intrinsic to the set of allowable first coronas (since the set of allowable first coronas list what boundaries are allowed to meet each other).
Since our definition of $\mathcal{X}$ and its atlas correspond to the first coronas of tiles in $T$, with $P_{i}^{s}$ replaced by $g_{s}(P_{j}^{s})$, any tiling by $\mathcal{X}$ is MLD to a tiling from $\mathcal{P}$.
Since $|\mathcal{X}_{s}| \leq |\mathcal{P}_{s}|$ then $|\mathcal{X}| \leq |\mathcal{P}|$. \hfill $\square$

\begin{cor}
Take a prototile set $\mathcal{P}$ and partition it into a set of equivalence classes $\mathcal{P}= \coprod \mathcal{P}_{s}$, $s \in \{1, \ldots, m\}$ as in the previous construction.
If there exists $\mathcal{P}_{s}$ such that $|\mathcal{X}_{s}| < |\mathcal{P}_{s}|$, there exists a prototile set (with atlas rules) which tiles $\mathbb{R}^{n}$ with less prototiles than $\mathcal{P}$.

\end{cor}

\begin{proof}
We know that $|\mathcal{X}_{s}| < |\mathcal{P}_{s}|$, thus $|\mathcal{X}| < |\mathcal{P}|$.
\end{proof}

\begin{remark}
This method of construction produces a prototile set with cardinality $\sum_{s=1}^{m} \lceil \frac{|\mathcal{P}_{s}|}{|G_{s}|} \rceil$.
\end{remark}

\begin{remark}
$\mathcal{P}$ is strongly aperiodic iff $\mathcal{X}$ is strongly aperiodic. (See \cite{STRAUSSSURVEY} for definition).
This is because every tiling in $\mathcal{X}$ is MLD to a tiling in $\mathcal{P}$, and strong aperiodicity is preserved under MLD equivalency.
\end{remark}

\section{Motivating examples and further improvements}

\begin{example}
For a simple illustration of the method, let us consider a tiling of the plane by 13 Wang tiles (unit squares with matching rules defined by matching coloured edges) as given in \cite{CULIK, KARI}.
Label the Wang tiles as $\{Q_{1}, \ldots, Q_{13} \}$. We can apply the above construction to get a function from $\{Q_{j}\}_{j=1}^{13}$ to $\{(P_{i}, r) | 1 \leq i \leq 2 , r \in D_{4}\}$, where $D_{4}$ is the group of symmetries of the square.

For example, enumerate the symmetries of the square as $\{r_{1}, r_{2}, \ldots, r_{8}\}$.
Then such a function could send $\{ Q_{j} | 1 \leq j \leq 8\}$ to $r_{j}(P_{1})$, and the remaining tiles $\{Q_{j} | 9 \leq j \leq 13\}$ to $r_{j-8}(P_{2})$.
The result is shown in diagram \ref{wangexample}, for a small patch of the tiling.

\end{example}

As is common with Wang tiles, the colouring of $\{Q_{j}\}_{j=1}^{13}$ is represented as actual colours superimposed onto the tile. We represent the change of prototile set from $\{Q_{j}\}_{j=1}^{13}$ to $\{P_{1}, P_{2}\}$ by adding a colouring to $P_{1}$ and $P_{2}$, which looks like their alphabetical symbols. This colouring has uncoloured points on the exterior of the tile (and thus no effect on the matching rules), but admits a free action by the symmetry group of a square.

\begin{figure}[!hbtp]
\includegraphics[angle=0, width=0.8\textwidth]{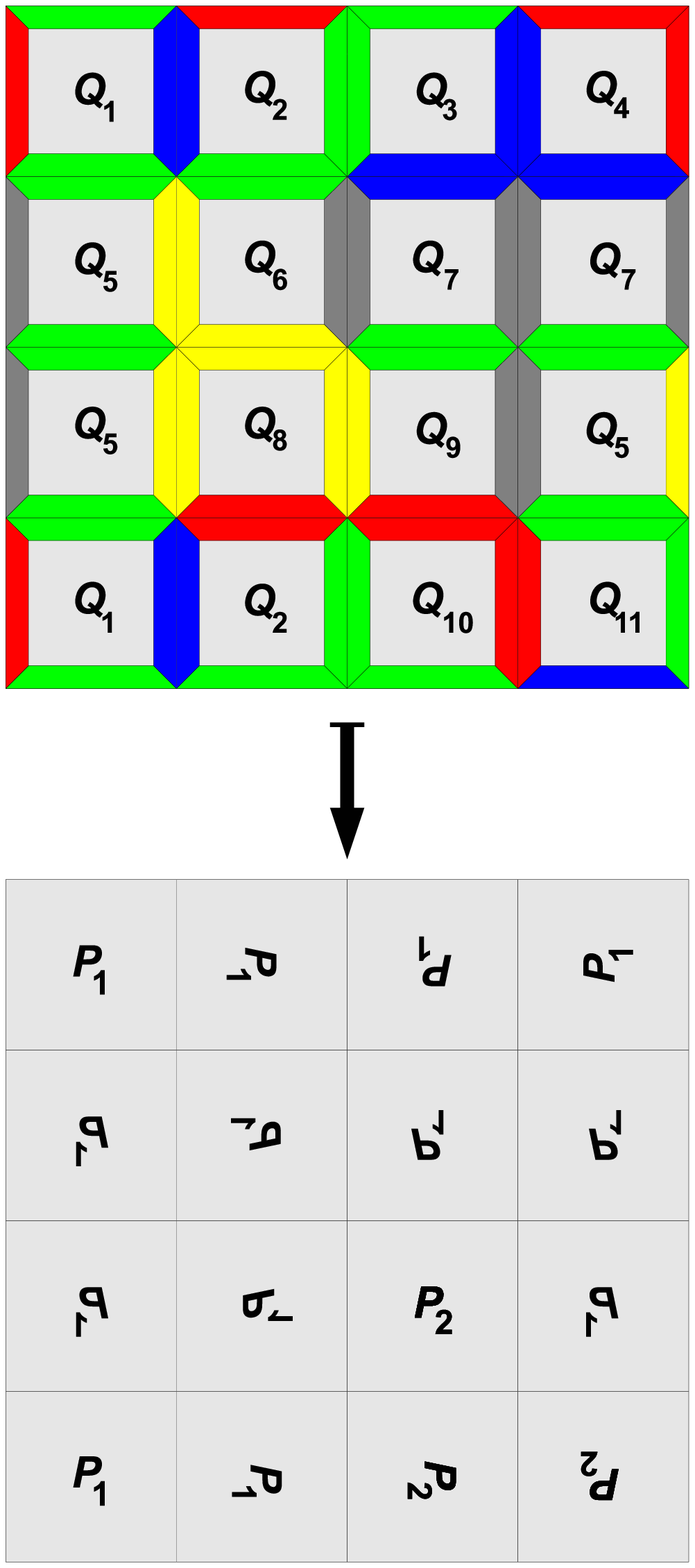}
\caption{Top picture shows a tiling with prototiles $Q_{1}, \ldots , Q_{13}$ with facet matching rules, and translation as an isometry group. \label{wangexample}
Bottom picture uses a two element prototile set, with rotations, reflections and translations as a isometry group.}
\end{figure}

\begin{example}
Consider Kari's Wang cube prototiles, $W$ \cite{KARICUBE}.
This is a set of 21 unit cube prototiles, with facet matching rules that only tile $\mathbb{R}^{3}$ aperiodically.

Choose a unit cube prototile $A$ which has an asymmetric label;
i.e., for any two distinct isometries of the cube $f,g \in O_{h}$, $f(A) \neq g(A)$.


Since the set of isometries of the cube is of cardinality 48, we can choose 21 unique isometries $i_{k} \in O_{h}$.
We use the method in Construction $1$ to replace $P_{k}$ with $i_{k}(A)$.

Thus we have an aperiodic protoset with one prototile which is MLD to Kari's Wang Cubes.
Note that we have lost the property of matching rules being determined on faces, and replaced them with a set of legal one corona patches (which cannot be rotated or reflected, of course). We have also had to broaden the set of allowable mappings of the prototiles into the tiled space, from translations to translations and rotation/reflections.
\end{example}

%

\begin{remark}
This algorithm can be further improved, by partitioning $\mathcal{P}$ into equivalence classes based on what prototiles have the same support \emph{up to isometry}, not just translation.

Let $T$ be a $(\mathcal{P},G_{Tr},g,r)$-tiling as in Construction $1$. If there is a prototile $P_{i} \in \mathcal{P}$ whose support is a non-trivial isometry of another prototile $P_{j}$ (where $i\neq j$), then the resulting $(\mathcal{X},G_{I},\mathcal{U})$-tiling may have less prototiles than one originating from Construction $1$.

\end{remark}

\begin{construction}
Partition $\mathcal{P}=\coprod \mathcal{P}^{s}$, $s=\{1,\ldots,p\}$ where $P_{i} \sim P_{j}$ iff $\emph{supp}(P_{i})=\emph{supp}(P_{j})$ up to the action of an element of $G_{I}$.

Further partition $\mathcal{P}^{s}=\coprod \mathcal{P}_{t}^{s}$, $t=\{1,\ldots,q\}$ where $P_{a}^{s} \sim P_{b}^{s}$ iff $\emph{supp}(P_{a}^{s})=\emph{supp}(P_{b}^{s})$ up to the action of an element of $G_{Tr}$.

This two-stage partitioning gives us a collection of equivalence classes ($\coprod_{s,t} \mathcal{P}_{t}^{s}$) as per the first construction. Additionally we know that there exist isometries in $G_{I}$ from elements of $\mathcal{P}_{i}^{s}$ to elements of $\mathcal{P}_{j}^{s}$.
Take the $\mathcal{P}_{t}^{s}$ with the largest cardinality and denote it $\mathcal{P}_{T}^{s}$. From the definition of $\mathcal{P}^{s}$ there exists an isometry $\alpha_{P_{i}P_{j}}$ such that $\alpha_{P_{i}P_{j}}(\emph{supp}(P_{i})) = supp(P_{j})$. Furthermore we know that an given isometry can only take elements from one set $\mathcal{P}_{i}^{s}$ to $\mathcal{P}_{T}^{s}$ (by definition of equivalence class).
Thus we can replace any prototile in $\mathcal{P}_{t}^{s}$ with a unique isometry of a prototile in $\mathcal{P}_{T}^{s}$, since $|\mathcal{P}_{t}^{s}| \leq |\mathcal{P}_{T}^{s}|$.

By applying the previous construction to $\mathcal{P}_{T}^{s}$, we can get a minimal uncoloured prototile set $\mathcal{X}^{s}$ that can be used to translate prototiles in $\mathcal{P}_{T}^{s}$, and hence $\mathcal{P}^{s}$, to atlas rules.

\end{construction}

\begin{example}
Take a prototile set $\mathcal{T}$ of equilateral triangles, as shown in figure 2. The prototiles have two different orientations, and three (could be up to six) colours.

You partition $\mathcal{T}$ into $\mathcal{T}=\mathcal{T}_{1}$, since all prototiles in $\mathcal{T}$ have the same support, up to isometry.
You then further partition $\mathcal{T}_{1}=\mathcal{T}_{1}^{1} \coprod \mathcal{T}_{1}^{2}$, where $\mathcal{T}_{1}^{1}$ is the set of prototiles with point upwards, and  $\mathcal{T}_{1}^{2}$ is the set of prototiles with point downwards. Denote the first prototile of $\mathcal{T}_{1}^{1}$ as $t_{1}$.
Applying the first construction to  $\mathcal{T}_{1}^{1}$ gives you $f(P_{i} \in P_{1}^{1})= d_{i}(t_{1})$, for $d_{i} \in D_{3}$, and $f(P_{i} \in P_{1}^{2})= \emph{rot}_{\frac{\pi}{3}}(d_{i}(t_{1}))$.
While this is sufficient to define the tiling, it has the problem that any picture of the tiling needs to include information about the isometries used for each tile. Thus we replace $t_{1}$ with a tile $x$ with an uncoloured boundary, but with a coloured interior which is not preserved under any non-identity element of $D_{3}$.

\begin{figure}[!hbtp]
\includegraphics[angle=90, width=0.8\textwidth]{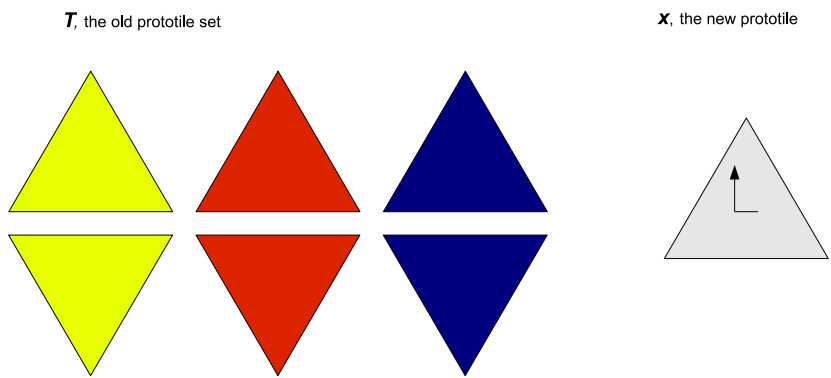}
\caption{New and old prototile set}
\end{figure}
\end{example}

\section{Acknowledgements}
Thanks are due to Chaim Goodman-Strauss, Joshua E. S. Socolar and Edmund Harriss for helpful conversations on this paper. My supervisor John Hunton has assisted considerably with improving the readability of this paper.
We also thank the University of Leicester and EPSRC for a doctoral fellowship. The results of this article will form part of the author's PhD thesis.

\end{document}